\documentclass[graybox]{svmult}
\usepackage{amssymb,amsfonts,amsmath}

\usepackage{mathptmx}       
\usepackage{helvet}         
\usepackage{courier}        
\usepackage{type1cm}        
\usepackage{makeidx}         
\usepackage{graphicx}        
\usepackage{multicol}        
\usepackage[bottom]{footmisc}

\makeindex             

\begin{document}

\title*{Geometry of rank 2 distributions with nonzero Wilczynski invariants and  affine control systems with one input}
\titlerunning { Rank 2 distributions with nonzero Wilczynski invariants and affine control systems}
\author{Boris Doubrov and Igor Zelenko}
\institute{Boris Doubrov
\at Belarussian State University, Nezavisimosti Ave.~4, Minsk 220030, Belarus;
\email {doubrov@islc.org} \and
Igor Zelenko
\at Department of Mathematics, Texas A$\&$M University,
   College Station, TX 77843-3368, USA; \email: {zelenko@math.tamu.edu}}
%
%

\maketitle


\renewcommand{\vec}{\overrightarrow}
\newcommand {\trans} {^{\,\mid\!\!\!\cap}}
\newcommand{\vf}{\varphi}
\newcommand{\mJ}{\mathcal J}
\newcommand{\e}{\varepsilon}
\newcommand{\dd}[1]{\frac{\partial}{\partial #1}}
\renewcommand{\theequation}{\thesection.\arabic{equation}}

\centerline{\emph{\large{Dedicated to Andrei Agrachev, our teacher and mentor,}}}
\centerline{\emph{\large{ on the occasion of his 60th birthday.}}}
\vskip .2in

\abstract{We demonstrate how the novel approach to the local geometry of structures of nonholonomic nature, originated by Andrei Agrachev, works in the following two situations: rank 2 distributions of maximal class in $\mathbb R^n$ with non-zero generalized Wilczynski invariants and  rank $2$ distributions
of maximal class in $\mathbb R^n$ with additional structures such  as affine control system with one input spanning these distributions,  sub-(pseudo)Riemannian structures etc.
The common feature of these two situations is that each abnormal extremal (of the underlying rank $2$ distribution) possesses a distinguished parametrization. This fact allows one to construct   the canonical frame on a $(2n-3)$-dimensional bundle in both situations for arbitrary $n\geq 5$ . The moduli spaces of the most symmetric models for both situations are described as well. The relation of our results to the divergence equivalence of Lagrangians of higher order is given.}

 
\section{Introduction}
  About seventeen years ago Andrei Agrachev proposed the idea to study the local geometry of control systems and geometric structures on manifolds by studying the flow of extremals of optimal control problems naturally associated with these objects \cite{agrachev,agrgam1, jac1}. Originally  he considered situations when  one can assign a curve of Lagrangian subspaces of a linear symplectic space or, in other words, a curve in a Lagrangian Grassmannian to an extremal of these optimal control problems. This curve was called the Jacobi curve of this extremal, because it contains all information about the solutions of the Jacobi equations along it. Agrachev's constructions of Jacobi curves worked in particular for normal extremals of sub-Riemannian structures and abnormal extremals of rank 2 distributions. Similar idea can be used for abnormal extremals of distribution of any rank, resulting in more general curves of coisotropic subspaces in a linear symplectic space \cite{doubzel3, quasi}.

The key point is that the differential geometry of the original structure can be studied via differential geometry of such curves with respect to the action of the linear symplectic group. The latter problem is simpler in many respects than the original one. In particular, any symplectic invariants of the Jacobi curves produces the invariant of the original structure.

This idea proved to be very prolific. For the geometry of distributions, first it led to a new geometric-control interpretation of the classical Cartan invariant of rank $2$ distributions on a five dimensional manifold, relating it to the classical Wilczynski invariants of curves in projective spaces \cite{zelcart, zelvar, zelnur}. It also gave a new effective method of the calculation of the Cartan tensor  and  the generalization of the latter invariant to rank 2 distributions on manifolds of arbitrary dimensions. These new invariants are obtained from the Wilczynski invariants of curves in projective spaces, induced from the Jacobi curves by a series of osculations together with the operation of taking skew symmetric complements. They are called the \emph{generalized Wilczynski invariants of rank 2 distributions} (see section \ref{unparsec} for details).

 Later on, we used this approach for the construction of the canonical frames for rank 2 distributions on manifolds of arbitrary dimension \cite{doubzel1, doubzel2}, and, in combination with algebraic prolongation techniques in a spirit of N. Tanaka, for the construction of the canonical frames for distributions of  rank $3$ \cite{doubzel3} and recently of arbitrary rank \cite{ quasi,jacsymb} under very mild genericity assumptions called maximality of class.  Remarkably, these constructions are independent of the nilpotent approximation (the Tanaka symbol) of a distribution at a point and even independent of its small growth vector.
 This extends significantly the scope of distributions for which  the canonical frames  can be constructed explicitly and in an unified way compared to the Tanaka approach (\cite{tan1, mori, aleks, zeltan}).

Perhaps the case of rank $2$ distributions of maximal class in $\mathbb R^n$ with $n>5$ provides the most illustrative example of the effectiveness of this approach, because the construction of the canonical frame in this case needs nothing more than some simple facts from the classical theory of curves in projective spaces such as the existence of the canonical projective structure on such curves, i.e. a special set of parametrizations defined up to a M\"{o}bius transformation (see section \ref{unparsec} below). The canonical frame for such distributions is constructed in a unified way on a bundle  of dimension  $2n-1$ and this dimension cannot be reduced, because there exists the unique, up to a local equivalence, rank 2 distribution of maximal class in $\mathbb R^n$ with the pseudo-group of local symmetries of dimension equal to $2n-1$.

For this most symmetric rank $2$ distribution of maximal class all generalized Wilczynski invariants are identically zero. However, if we assume that at least one generalized Wilczynski invariant does not vanish, then one would expect that the canonical frame can be constructed on a bundle of smaller dimension.
 In this case  the canonical parametrization, up to a shift, on abnormal extremals can be distinguished  instead of the canonical projective structure.

 Similarly, the canonical parametrization, up to a shift, on abnormal extremals can be distinguished in the case of a rank $2$ distribution $D$  with the additional structures
 defining a control system with one input satisfying certain regularity assumptions. A \emph{ control system with one input on a distribution $D$} is given by choosing  a  one-dimensional submanifold  $\mathcal V_q$ on each fiber $D(q)$ of the distribution $D$ (smoothly depending on $q$).

 \begin{definition}
 \label{reglinedef}
   The set $\mathcal V_q$ at a point $q$  is  called the \emph{set of admissible velocities of the control system at $q$.}  A line in $D(q)$ (through the origin) intersecting  the set $\mathcal V_q\backslash \{\text{the origin of }D(q)\}$ in a finite number of points is called a \emph{regular line} of the control system at the point $q$.
 \end{definition}

 \begin{definition}
 \label{regcontdef}
 We say that a control system with one input on a rank $2$ distribution $D$ is \emph{regular} if  for any point $q$ the sets of regular lines
 is a nonempty open subset of the projectivization $\mathbb P D(q)$ .
 \end{definition}

 An important particular class of examples of such control systems  is when $\mathcal V_q$ is an affine line. In this case we get an \emph{affine
 control system with one input and with a non-zero drift}. Another examples are sub-(pseudo)Riemannian structures, when the curves are $\pm1$-level sets of non-degenerate quadrics. For  affine control systems with a non-zero drift and sub-Riemannian structures all lines in $D(q)$ are regular, while for sub-pseudo-Riemannian case all lines except the asymptotic lines of the quadrics are regular.

 The goal of this paper is to demonstrate the approach, originated by Andrei Agrachev, in these two simplified but still important situations: of rank $2$ distributions of maximal class with at least one nonvanishing generalized Wilczynski invariant and of regular control system with one input on rank $2$ distributions of maximal class.
 We show that \emph{in both situations the canonical frame  can be constructed in a unified way on a bundle of dimension  $2n-3$ for all $n\geq 5$} (Theorem \ref{theor1par}, section \ref{canfrparsec} ).
We also describe in both situations all models with the pseudo-group of local symmetries of dimension $2n-3$. i.e. the most symmetric ones, among the considered class of objects (Theorems \ref{symmodthm0} and \ref{symmodthm1} below and their reformulation in Theorem \ref{symmodthm90} and \ref{symmodthm91}, section \ref{symmodsec}).

  The most symmetric models of two considered situations are closely related. In both situations they are not unique and depend on continuous parameters. Let us describe these models. Given a tuple of $n-3$ constants $(r_1,\ldots, r_{n-3})$ let $A_{(r_1,\ldots,r_{n-3})}$ be the following affine control system in $R^n$ taken with coordinates
  $(x,y_0,\ldots, y_{n-3},z)$:
  \begin{equation}
\label{affinemax1}
\dot q=X_1\bigl(q\bigr)+u X_2(q),
\end{equation}
where
\begin{align}
X_1 &=\dd{x} + y_1\dd{y_0} + \dots + y_{n-3}\dd{y_{n-4}} +\nonumber\\
~&~\bigl(y_{n-3}^2+r_1 y_{n-4}^2 +r_2 y_{n-5}^2+\ldots r_{n-3} y_0^2\bigr)\dd{z},\label{X1}\\
X_2 &= \dd{y_{n-3}}\label{X2}.
\end{align}
and denote by $D_{(r_1,\ldots,r_{n-3})}$ the corresponding rank 2 distribution generated by the vector fields $X_1$ and $X_2$ as in \eqref{X1}-\eqref{X2}.
In the case of regular control systems we prove the following


\begin{theorem}
\label{symmodthm0}
If a  regular  control systems
with one input
on a rank 2 distribution of maximal class in $\mathbb R^n$ with $n\geq 5$ has a group of local symmetries of dimension $2n-3$, then it is locally equivalent to the system $A_{(r_1,\ldots,r_{n-3})}$
 for some constants $r_i\in\mathbb R$, $1\leq i\leq n-3$. The affine control systems $A_{(r_1,\ldots,r_{n-3})}$ corresponding to the different tuples $ (r_1,\ldots,r_{n-3})$  are not equivalent.
\end{theorem}

In other words, the map $(r_1,\ldots, r_{n-3})\mapsto A_{(r_1,\ldots,r_{n-3})}$  identifies the space $\mathcal{A}_n$ of the most symmetric, up to a local equivalence, regular control systems
on rank 2 distributions of maximal class in $\mathbb R^n$ with $\mathbb R^{n-3}$.

Further, in the space $\mathcal A_n$ there is a special 1-foliation $\mathcal F$ (i.e a foliation by curves) with a singularity at the origin (under the identification of $\mathcal A_n$ with $\mathbb R^{n-3}$ given by Theorem \ref{symmodthm0}) such that the rank 2 distributions corresponding to the affine systems from the same leaf of $\mathcal F$ are locally equivalent and the rank 2 distributions corresponding to the affine systems from the different leaves of $\mathcal F$  are not equivalent. Among all leaves of $\mathcal F$  there is an exceptional leaf $\mathcal F(0)$ passing through the origin and the rank 2 distributions corresponding to the affine systems from this leaf are locally equivalent to the most symmetric rank 2 distribution in $\mathbb R^n$ of maximal class. It turns out that
the space of the most symmetric rank $2$ distributions of maximal class in $\mathbb R^n$  with nonzero Wilczynski invariants can be identified with the quotient space of $\mathcal A_n\backslash \mathcal F(0)$ by the foliation $\mathcal F$.

In more details,
first, as shown in \cite{doubzel1, doubzel2}, the most symmetric rank 2 distribution in $R^n$ of maximal class with $n\geq 5$ is locally equivalent
to $D_{(0,\ldots,0)}$. It turns out that among all distributions of the type $D_{(r_1,\ldots,r_{n-3})}$ there is a one-parametric family of distribution which are locally equivalent to $D_{(0,\ldots,0)}$. To describe this family we need the following definition

\begin{definition}
\label{exceptdef}
The tuple of $m$ numbers $(r_1,\ldots,r_{m})$ is called exceptional if
the roots of the polynomial
\begin{equation}
\label{charpoly}
\lambda^{2m}+\sum_{i=1}^m (-1)^i r_i\lambda^{2(m-i)}
\end{equation}
constitute an arithmetic progression (with the zero sum in this case). Equivalently, $(r_1,\ldots,r_{m})$ is exceptional if
$r_i=\alpha_{m,i}\left(\frac{r_1}{\alpha_{m,1}}\right)^i$, $1\leq i\leq m$, where the constants $\alpha_{m,i}$, $1\leq i\leq m$, satisfy the following identity
\begin{equation}
\label{c_iprop}
x^{2m}+\sum_{i=1}^m (-1)^i\alpha_{m,i} x^{2(m-i)}=\prod_{i=1}^{m}\bigl(x^2-(2i-1)^2\bigr).
\end{equation}
\end{definition}

It turns out (Corollary  \ref{D0cor}, subsection \ref{wilczsubsec})
that \emph{the distribution $D_{(r_1,\ldots,r_{n-3})}$ is locally equivalent to the distribution $D_{(0,\ldots,0)}$ (or, equivalently, has the algebra of infinitesimal symmetries of the maximal possible dimension among all rank $2$ distributions of maximal class in $\mathbb R^n$) if and only if the tuple
$(r_1,\ldots,r_{n-3})$ is exceptional in the sense of Definition \ref{exceptdef}}.  As far as we know, this simple but nice observation was not mentioned in the existing literature. This observation is based on the following simple fact from the representation theory of the Lie algebra $\mathfrak{sl}_2$: the spectrum of any element in the image of an irreducible representation of $\mathfrak{sl}_2$ forms an arithmetic progression (see Proposition \ref{Wilczexceptrem}, subsection \ref{wilczsubsec}).

The analog of Theorem \ref{symmodthm0} for rank $2$ distributions of maximal class in $\mathbb R^n$  with nonzero Wilczynski invariants can be formulated as follows

\begin{theorem}
\label{symmodthm1}
 If a rank $2$ distribution in $R^n$ of maximal class, $n\geq 5$, with at least one nowhere vanishing generalized Wilczynski invariant has a group of local symmetries of dimension $2n-3$, then it is locally equivalent to the distribution $D_{(r_1,\ldots,r_{n-3})}$ where the tuple $(r_1,\ldots,r_{n-3})$ is not exceptional in the sense of Definition \ref{exceptdef}.
 Two distributions $D_{(r_1,\ldots,r_{n-3})}$ and $D_{(\tilde r_1,\ldots,\tilde r_{n-3})}$ are locally equivalent if and only if
\begin{equation}
\label{rhotrans}
\text{\rm{there exists }} c\neq 0 \text{\rm{ such that }} \tilde r_i=c^{2i} r_i, \quad 1\leq i\leq n-3.
\end{equation}
\end{theorem}

The aforementioned foliation $\mathcal F$ on the space $\mathcal A_n$ can be described as follows: the exceptional leaf $\mathcal F(0)$ consists of the exceptional tuples  in the sense of Definition \ref{exceptdef}; other leaves are exactly the equivalence classes on  $\mathcal A_n\backslash \mathcal F(0)$ with respect to the equivalence relation given by \eqref{rhotrans}. Note that the exceptional leaf ${\mathcal F}(0)$ is also the union of three equivalence classes with respect to the same equivalence relation, one of which is the origin.

The above symmetric models of distributions are associated with the following underdetermined ordinary differential equations (Monge equations)
\begin{equation}
\label{under}
z^{\,\prime}(x)
=\bigl(y^{(n-3)}(x)\bigr)^2+r_1 \bigl(y^{(n-4)}(x)\bigr)^2 +\ldots r_{n-3} y^2(x),
\end{equation}
and also with  the Lagrangians
\begin{equation}
\label{Lagr}
\int\Bigr(\bigl(y^{(n-3)}(x)\bigr)^2+r_1 \bigl(y^{(n-4)}(x)\bigr)^2 +\ldots r_{n-3} y^2(x)\Bigl)\,dx,
\end{equation}
which are quadratic with respect to the derivatives and have constant coefficients.
It is well known (see \cite{olver}, discussion in the beginning of p. 242 there) that these Lagrangians are the most symmetric ones
among all Lagrangians $$\displaystyle{\int F\bigl(x,y(x),\ldots, y^{(n-3)}(x)\bigr)\, dx}$$ with $F_{y^{(n-3)} y^{(n-3)}}\neq 0$, up to a contact transformation and modulo divergence.
Our results here together with the relation of this equivalence problem and its modification to the equivalence of very special rank 2 distributions of maximal class studied in \cite{var} give an alternative proof of this fact.
Note that for these most symmetric Lagrangians  the Euler-Lagrange equation is a linear equation with constant coefficients (such that its characteristic polynomial coincides with the polynomial in \eqref{charpoly} where $m=n-3$).

Note that for rank $2$ distributions in $\mathbb R^5$ the notion of maximality of class coincides with the condition that the small growth vector is equal to $(2,3,5)$. As was shown by \`{E}lie Cartan in his famous paper \cite{cart10} the most symmetric distribution among all distributions with small growth vector $(2,3,5)$ has 14 dimensional algebra of infinitesimal symmetries and this distribution is the unique distribution with identically zero Cartan invariant (which coincides with the (unique in this case) generalized Wilczynski invariant). Therefore in this case our approach gives the unified construction of the canonical frame on a $7$-dimensional bundle for all distributions with the small growth vector $(2,3,5)$ except the most symmetric one, which provides also an alternative way to get the Cartan classification of submaximal symmetric models for these distributions \cite{cart10} [chapter IX]. Note also that in the case $n=5$ the construction of the canonical frame was already done in the PhD thesis of the second author \cite[subsection 10.5]{phd}. It is worth to mention that for $n=5$ an alternative way to describe these submaximal models is via the family of underdetermined ODEs  (Monge equations) $z'(x)=(y''(x))^\alpha$ with $\alpha\notin\{-1, \frac{1}{3},\frac{2}{3},2\}$
(see, for example, \cite{nur}[Example 6],\cite{krug}[section 5]).


Finally note that regarding regular control systems with one input on a rank 2 distribution of maximal class  the obtained models \eqref{affinemax1}-\eqref{X2}  are maximally symmetric. Affine control systems with one input were considered also in \cite{zelclass}, but the genericity assumptions imposed  there are much stronger than our genericity assumptions here.

The paper is organized as follows. The main results are given in sections \ref{canfrparsec} and \ref{symmodsec} (Theorem \ref{theor1par} and   Theorems \ref{symmodthm90}-\ref{symmodthm91}, which are reformulations of Theorems \ref{symmodthm0}-\ref{symmodthm1} above). Sections \ref{abnsec}-\ref{londonsec} are preparatory for section \ref{canfrparsec}, section \ref{sympsec} is preparatory for section \ref{symmodsec}.  In sections \ref{abnsec}-\ref{unparsec} we list all necessary facts about abnormal extremals of rank 2 distributions, their Jacobi curves and the invariants of unparametrized curves in projective spaces. The details can be found in \cite{doubzel1, zelvar, zelrigid}. In section \ref{londonsec} we summarize the main results of \cite{doubzel1, doubzel2} about canonical frames for rank 2 distributions of maximal class in order to compare them with the analogous results of sections  \ref{canfrparsec} and \ref{symmodsec}. In section \ref{sympsec} we list all necessary facts about  the invariants of parametrized self-dual curves in projective spaces.

%



\section {Abnormal extremals of rank 2 distributions}
\label{abnsec}
\setcounter{equation}{0}

Let $D$ be a rank 2 distribution on a manifold $M$. A smooth section of a vector bundle $D$ is called a \emph{horizontal vector field of $D$}. Taking iterative brackets of horizontal vector fields of $D$, we obtain the natural filtration $\{\dim D^j(q)\}_{j\in \mathbb N}$  on each tangent space $T_q M$.  Here $D^j$ is the $j$-th power
of the distribution $D$, i.e., $D^j=D^{j-1}+[D,D^{j-1}]$, $D^1=D$, or , equivalently, $D^j(q)$ is a linear span of all Lie brackets  of the length not greater than $j$ of horizontal vector fields of $D$ evaluated at $q$.

Assume that $\dim D^2(q)=3$ and $\dim D^3(q)>3$ for any
$q\in M$. Denote by $(D^j)^{\perp}\subset T^*M$ the
annihilator of the $j$th power $D^j$, namely
$$(D^j)^{\perp}= \{(p,q)\in T^*M:\,\, p\cdot v=0\,\,\forall
v\in D^j(q)\}.$$

Recall that abnormal extremals of $D$ are by definition the Pontryagin extremals with the vanishing Lagrange multiplier near the
functional for any extremal problem with constrains, given by the distribution $D$. They depend only on  the distribution $D$ and
not on a functional.

It is easy to show (see, for example, \cite{zelrigid,doubzel2}) that for rank 2 distributions all abnormal extremals lie in $(D^2)^\perp$ and that through any point of the codimension $3$ submanifold
$(D^2)^\perp\backslash(D^3)^\perp$ of $T^*M$ passes exactly one abnormal extremal or, in other words, $(D^2)^\perp\backslash(D^3)^\perp$ is foliated by
the characteristic $1$-foliation of abnormal extremals. To describe this foliation let
$\pi:T^*M\mapsto M$ be the canonical projection. For any
$\lambda\in T^*M$, $\lambda=(p,q)$, $q\in M$, $p\in
T_q^*M$, let $\mathfrak{s}(\lambda)(\cdot)=p(\pi_*\cdot)$
be the canonical Liouville form and $\sigma=d\mathfrak {s}$
be the standard symplectic structure on $T^*M$. Since the
submanifold $(D^2)^\perp$ has odd codimension in $T^*M$,
the kernels of the restriction $\sigma|_{(D^2)^\perp}$ of $\sigma$
on $(D^2)^\perp$ are not trivial. At the
the points of $(D^2)^\perp\backslash (D^3)^\perp$ these
kernels are one-dimensional.
They form the \emph{characteristic line distribution} in $(D^2)^\perp\backslash(D^3)^\perp$, which will be denoted by
${\mathcal C}$. The line distribution ${\mathcal C}$ defines the desired  \emph{characteristic 1-foliation} on
$(D^2)^\perp\backslash(D^3)^\perp$ and the leaf of this foliation through a point  is exactly the abnormal extremal passing through this point.
From now on we shall work with abnormal extremals which are integral curves of the characteristic distribution $\mathcal C$.
%

The characteristic line distribution
${\mathcal C}$ can be easily described in terms of a local basis  of
the distribution $D$, i.e. $2$ horizontal vector fields $X_1$ and $X_2$ such that
$D( q)={\rm span}\{X_1(q), X_2(q)\}$ for all $q$ from some open set of $M$.
Denote by
\begin{equation}
\label{x345}
X_3=[X_1,X_2],\,\,
X_4 =\bigl[X_1,[X_1,X_2]\bigr],\,\,
X_5 =\bigl[X_2,[X_1,X_2]\bigr].
\end{equation}
Let us introduce the ``quasi-impulses''
$u_i:T^*M\mapsto \mathbb R$, $1\leq i\leq 5$,
\begin{equation}
\label{quasi25} u_i(\lambda)=p\cdot X_i(q),\,\,\lambda=(p,q),\,\,
q\in M,\,\, p\in T_q^* M.
\end{equation}
Then by the definition
\begin{equation}
\label{d2u}
(D^2)^\perp=\{\lambda\in T^*M:
u_1(\lambda)=u_2(\lambda)=u_3(\lambda)=0\}.
\end{equation}
 As usual, for
a given function $h:T^*M\mapsto \mathbb R$ denote by $\overrightarrow h$
the corresponding Hamiltonian vector field defined by the
relation $i_{\overrightarrow h}\sigma
=-d\,G
$.
Then by the direct computations (see, for example, \cite{doubzel2})
the characteristic line distribution $\mathcal C$ satisfies
\begin{equation}
\label{foli25}
{\mathcal C}= \rm{span} \{u_4\overrightarrow{u}_2-u_5\overrightarrow{u}_1 \}.
\end{equation}


\section{Jacobi curves of abnormal extremals}
\label{jacsec}
\setcounter{equation}{0}

Now we are ready to define the Jacobi curve of an abnormal extremal of $D$.
For this first lift the distribution $D$ to $(D^2)^\perp$, namely considered the distribution $\mJ$ on $(D^2)^\perp$ such that
\begin{equation}
\label{prejac}
{\mathcal J}(\lambda)=
\{v\in T_{\lambda}(D^2)^\perp:\,d\pi (v)\in D(\pi\bigl(\lambda)\bigr)\}.
\end{equation}
 Note that $\dim \mJ = n-1$ and $\mathcal C\subset \mJ$ by (\ref{foli25}) .
The distribution ${\mathcal J}$ is called the \emph {lift of the
distribution $D$ to $(D^2)^\perp\backslash(D^3)^\perp$}.

Given a segment  $\gamma$ of an abnormal extremal (i.e. of a leaf of the $1$-characteristic foliation) of
$D$, take  a sufficiently small neighborhood $O_\gamma$  of $\gamma$ in $(D^2)^\perp$
such that the quotient
$N=O_\gamma /(\text {\emph{the characteristic one-foliation}})$
 is a
well defined smooth manifold. The quotient manifold $N$ is a
symplectic manifold endowed with the symplectic structure
$\bar\sigma$ induced by $\sigma |_{(D^2)^\perp}$. Let

\begin{equation}
\label{phi}
\phi
:O_\gamma\to N
\end{equation}
be the canonical projection on the factor.
Define the following curves of subspaces in
$T_\gamma N$:
\begin{equation}
\label{jacurve}
\lambda\mapsto
\phi_*\bigl({\mathcal J}(\lambda) \bigr), \quad
\forall
\lambda\in\gamma.
\end{equation}
Informally speaking, these curves describe the dynamics of the distribution $\mathcal J$
w.r.t. the characteristic $1$-foliation along the abnormal extremal
$\gamma$.

Note that there exists a straight line, which is common to all subspaces
appearing in (\ref{jacurve}) for any $\lambda\in\gamma$. So, it is more
convenient to get rid of it by a factorization. Indeed, let $e$ be the Euler
field on $T^*M$, i.e., the infinitesimal generator of homotheties on the fibers
of $T^*M$. Since a transformation of $T^*M$, which is a homothety on each fiber
with the same homothety coefficient, sends abnormal extremals to abnormal
extremals, we see that the vector $\bar e= \phi_*e(\lambda)$ is the same for any
$\lambda\in \gamma$ and lies in any subspace appearing in (\ref{jacurve}). Let
\begin{equation}
\label{jacurve1} J_\gamma(\lambda)=\phi_*\bigl({\mathcal
J}(\lambda)\bigr)/\{\mathbb R \bar e\}, \quad \forall \lambda\in\gamma
\end{equation}

The (unparametrized) curve $\lambda\mapsto J_\gamma(\lambda),\,\lambda\in\gamma$ is called the \emph{Jacobi curve of the abnormal extremal $\gamma$}.
It is clear that all subspaces appearing in
(\ref{jacurve1}) belong to the space
\begin{equation}
\label{spaceW}
 W_\gamma=\{v\in T_\gamma N: \bar\sigma(v,\bar
e)=0\}/\{\mathbb R \bar e\}.
\end{equation}
 and that
\begin{equation}
\label{dimJ}
\dim J_\gamma(\lambda)=n-3.
\end{equation}
The space $W_\gamma$ is endowed with the natural symplectic structure $\tilde\sigma_\gamma$
induced by $\bar\sigma$.
Also $\dim W_\gamma=2(n-3)$.

Given a subspace $L$ of $W_\gamma$ denote by $L^\angle$ the skew-orthogonal complement of $L$ with respect to the symplectic form $\tilde\sigma_\gamma$,
$L^\angle=\{v\in W_\gamma, \sigma_\gamma(v,\ell)=0\quad \forall \ell\in L\}$. Recall that the subspace $L$ is called \emph{isotropic} if $L\subseteq L^\angle$,
\emph{coisotropic} if $L^\angle\subseteq L$, and \emph{Lagrangian}, if $L=L^\angle$. Directly from the definition, the dimension of an isotropic subspace does not exceed $\frac{1}{2}\dim W_\gamma$, and a Lagrangian subspace is an isotropic subspace of the maximal possible dimension $\frac{1}{2}\dim W_\gamma$.
The set of all Lagrangian subspaces of $W_\gamma$ is called the \emph{Lagrangian Grassmannian of $W_\gamma$}.

It is easy to see (\cite{doubzel2, zelvar}) that the Jacobi curve of an abnormal extremal consists of Lagrangian subspaces, i.e. it is a curve in the Lagrangian Grassmannian of $W_\gamma$. In the case $n\geq 5$ (equivalently, $\dim \,W_\gamma\geq 4$) curves in the Lagrangian Grassmannian of $W_\gamma$ have a nontrivial geometry with respect to the action of the linear symplectic group and any symplectic invariant of Jacobi curves of abnormal extremals produces an invariant of the original distribution $D$.

\section{Reduction to geometry of curves in projective spaces}
\label{projsec}
\setcounter{equation}{0}

In the earlier works \cite{jac1, zelvar} invariants of Jacobi curves were constructed using the notion of the cross-ratio of four points in Lagrangian Grassmannians analogous to the classical cross-ratio of four point in a projective line. Later, we developed a different method, leading to the construction of canonical bundles of moving frames and invariants for quite general curves in Grassmannians and flag varieties \cite{flag1,flag2}. The geometry of Jacobi curves $J_\gamma$ in the case of rank 2 distributions can be reduced to the geometry of the so-called self-dual curves in the projective space $\mathbb P W_\gamma$.

For this first one can produce a curve of flags of isotropic/coisotropic subspaces of $W_\gamma$ by a series of osculations together with the operation of taking skew symmetric complements.
 For this, denote by $C(J_\gamma)$ the \emph{tautological bundle} over $J_\gamma$: the fiber of $C(J_\gamma)$ over the point $J_\gamma(\lambda)$ is the linear space $J_\gamma(\lambda)$. Let $\Gamma(J_\gamma)$ be the space of all smooth sections of $C(J_\gamma)$. If $\psi:(-\varepsilon,\varepsilon)\mapsto \gamma$ is a parametrization of
$\gamma$ such that $\psi(0)=\lambda$,
then for any $i\geq 0$ define
\begin{eqnarray}
~&
J_\gamma^{(i)}(\lambda):={\rm span}\{\frac{d^j}{d\tau^j}\ell\bigl(\psi(t))\bigr|_{t=0}: \ell\in\Gamma(J_\gamma), 0\leq j\leq i\} \label{exti}\\
~&J_\gamma^{(-i)}(\lambda)= \bigl(J^{(i)}_\gamma(\lambda)\bigr)^\angle \label{contri}
\end{eqnarray}
For $\,i>0$ we say that the space $J^{(i)}_\gamma(\lambda)$ is the \emph{$i$-th osculating space of the curve $J_\gamma$ at $\lambda$}.

Note that $J_\gamma =J_\gamma^{(0)}$.
Directly from the definitions the subspaces  $J^{(i)}_\gamma(\lambda)$ are coisotropic for $i>0$ and isotropic for $i<0$ and the tuple $\{J^{(i)}_\gamma(\lambda)\}_{i\in\mathbb Z}$ defines a filtration of $W_\gamma$. In other words, the curve $\lambda\mapsto \{J^{(i)}_\gamma(\lambda)\}_{i\in\mathbb Z}$ is a curve of flags of $W_\gamma$. Besides, it can be shown \cite{zelvar}
that
$$\dim \,J^{(1)}(\lambda)-\dim\, J^{(0)}(\lambda) =\dim \,J^{(0)}(\lambda)-\dim\, J^{(-1)}(\lambda)=1,$$ which in turn implies that
$\dim \,J^{(i)}(\lambda)-\dim\, J^{(i-1)}(\lambda) \leq 1$, i.e. the jump of dimensions between the consecutive subspaces of the filtration
$\{J^{(i)}_\gamma(\lambda)\}_{i\in\mathbb Z}$ is at most $1$. This together with \eqref{dimJ} implies that $\dim \,J_\gamma^{(i)}(\lambda)\leq n-3+i$ for $i>0$.

We say that $\lambda$ is a \emph{regular point of $(D^2)^\perp\backslash (D^3)^\perp$} if  $\dim \,J_\gamma^{(i)}(\lambda)= n-3+i$ for $0<i\leq n-3$ or,  equivalently, if $ J_\gamma^{(n-3)}(\lambda)=W_\gamma$. A rank $2$ distribution $D$ is called of \emph{maximal class at a point $q\in M$} if at least one point in $\pi^{-1}(q)\cap (D^2)^\perp$ is regular. Since by \eqref{foli25} the characteristic distribution $\mathcal C$ generated by a vector field depending algebraically on the fibers $(D^2)^\perp$, if $D$ is of maximal class at a point $q\in M$, then the set of all regular points of $\pi^{-1}(q)\cap (D^2)^\perp$ is non-empty open set in Zariski topology. The same argument is used to show that the set of germs of rank 2 distributions of maximal class is generic.

If  $D$ is of maximal class at $q$ and $n\geq 5$, then by necessity $\dim D^3(q)=5$. The following question is still open: Does there  exist a rank $2$ distribution with
$\dim D^3=5$ such that it is not of maximal class on some open set of $M$? We proved that the answer is negative for $n\leq 8$ and we have strong evidences that the answer is negative in general.

\begin{remark}
\label{openrem}
Note that from \eqref{foli25} it follow that if a rank 2 distribution $D$ is of maximal class at a point $q\in M$ then the set of all lines
$\{d\pi \bigl(\mathcal C(\lambda)\bigr): \lambda\in \mathcal R_D\cap \pi^{-1}(q)\}$ is an open and dense subset of the projectivization $\mathbb P D(q)$
of the plane $D(q)$, where, as before, $\pi:T^*M\rightarrow M$ is the canonical projection. $\Box$
\end{remark}

From now on we will work with rank $2$ distributions of maximal class. In this case $\dim J_\gamma ^{(4-n)}(\lambda)=1$, i.e. the curve $J_\gamma ^{(4-n)}$ is a curve in the projective space $\mathbb PW_\gamma$. Moreover, the curve of flags $\lambda\mapsto \{J^{(i)}_\gamma(\lambda)\}_{i=3-n}^{n-3},\,\lambda\in\gamma$ is the curve of complete flags and the space $J^{(i)}_\gamma(\lambda)$ is the $(i+n-4)$th-osculating space of the curve $J_\gamma ^{(4-n)}$. In other words, the whole curve of complete flags $\lambda\mapsto \{J^{(i)}_\gamma(\lambda)\}_{i=3-n}^{n-3},\,\lambda\in\gamma$ can be recovered from the curve $J_\gamma ^{(4-n)}$ and the differential geometry of Jacobi curves of abnormal extremals of rank $2$ distributions is reduced to the differential geometry of curves in projective spaces.

\section{Canonical projective structure and Wilczynski invariants}
\label{unparsec}
\setcounter{equation}{0}

The differential geometry of curves in projective spaces is the
classical subject, essentially completed already in 1905 by
E.J.~Wilczynski (\cite{wilch}). In particular, it is well known that
these curves  are endowed with the canonical projective structure,
i.e., there is a distinguished set of parameterizations (called
projective) such that the transition function from one such
parametrization to another is a M\"{o}bius transformation. Let us
demonstrate how to construct it for the curve $\lambda\mapsto
J_\gamma^{(4-n)}(\lambda)$, $\lambda\in\gamma$.

As before, let $C(J_\gamma^{(4-n)})$
be the tautological bundle  $C(J_\gamma^{(4-n)})$ over $J_\gamma^{(4-n)})$.
Set $m=n-3$. Here we use a ``naive approach'', based on reparametrization rules
for certain coefficient in the expansion of the derivative of order $2m$ of certain sections
of  $C(J_\gamma^{(4-n)})$ w.r.t. to the lower order derivatives of this sections.
For the more algebraic point of view, based on Tanaka-like theory of curves of flags and $\mathfrak{sl}_2$-representations see
\cite{doub3, flag1}.

Take some parametrization
$\psi\colon I\mapsto \gamma$ of $\gamma$, where $I$ is an interval in $\mathbb R$
By above, for any section
$\ell$ of $C(J_\gamma^{(4-n)})$ one has that
\begin{equation}
\label{spanall} {\rm span}\bigl\{\frac{d^j}{dt^j}\ell\bigl(\psi(t)\bigr)\mid
0\leq j\leq 2m-1\bigr\}=W_\gamma.
\end{equation}
A curves in the projective space $\mathbb P W_\gamma$ satisfying the last property is called \emph{regular} (or \emph{convex}).
It is well known that there exists the unique, up to the
multiplication by a nonzero constant, section
$E$ of $C(J_\gamma^{(4-n)})$, called a \emph{canonical section of $C(J_\gamma^{(4-n)})$ with respect to the parametrization $\psi$}, such that
\begin{equation}
\label{last0}
\frac{d^{2m}}{dt^{2m}}E\bigl(\psi(t)\bigr)=\sum_{i=0}^{2m-2}B_i(t)\frac{d^{i}}{dt^{i}}
E\bigl(\psi(t)\bigr),
\end{equation}
i.e. the coefficient of the term
$\frac{d^{2m-1}}{dt^{2m-1}}E\bigl(\psi(t)\bigr)$ in the linear
decomposition of $\frac{d^{2m}}{dt^{2m}}E\bigl(\psi(t)\bigr)$ w.r.t. the
basis $\bigl\{\frac{d^i}{dt^i}E\bigl(\psi(t)\bigr):0\leq i\leq 2m-1\bigr\} $
vanishes.

Further, let $\psi_1$ be another parameter, $\widetilde E$ be a canonical section of $C(J_\gamma^{(4-n)})$ with respect to the parametrization $\psi_1$, and $\upsilon=\psi^{-1}\circ\psi_1$. Then directly from the definition it easy to see that

\begin{equation}
\label{EE}
\widetilde E\bigl(\psi_1(\tau)\bigr)=c(\upsilon'(\tau))^{\frac{1}{2}-m} E(\psi(t))
\end{equation}
for some non-zero constant $c$.

Now let $\widetilde B_{i}(\tau)$ be the coefficient in the linear decomposition of $\frac{d^{2m}}{d\tau^{2m}}\widetilde E\bigl(\psi_1(\tau)\bigr)$ w.r.t. the
basis $\bigl\{\frac{d^i}{d\tau^i}\widetilde E\bigl(\psi_1(t)\bigr):0\leq i\leq 2m-1\bigr\}$ as in \eqref{last0}. Then, using the relation \ref{EE}
it is not hard to show
that the coefficients $B_{2m-2}$ and
$\widetilde B_{2m-2}$ in the decomposition (\ref{last0}),
corresponding to parameterizations $\psi$ and $\psi_1$, are
related as follows:
\begin{equation}
 \label{rhorep}
 \widetilde
 B_{2m-2}(\tau)=\upsilon'(\tau)^2B_{2m-2}(\upsilon(\tau))-\frac{m(4m^2-1)}{3}
\mathbb{S}(\upsilon)(\tau),
\end{equation}
 where $\mathbb{S}(\upsilon)$ is the
Schwarzian derivative of $\upsilon$,
$\mathbb S(\upsilon)=
\frac {d}{dt}\Bigl(\frac {\upsilon''}{2\,\upsilon'}\Bigr)
-\Bigl(\frac{\upsilon''}{2\,\upsilon'}\Bigr)^2$.

From the last formula and the fact that ${\mathbb S}\upsilon\equiv
0$ if and only if the function $\upsilon$ is M\"{o}bius it follows
that \emph{the set of all parameterizations $\varphi$ of $\gamma$
such that
\begin{equation}
\label{A2m-2} B_{2m-2}\equiv 0
\end{equation}
defines the canonical projective structure on $\gamma$}. Such
parameterizations are called the \emph {projective parameterizations
of the abnormal extremal $\gamma$}. If $\psi$ and $\psi_1$ are two projectivization, then there exists a  M\"{o}bius transformation $\upsilon$ such that $\psi_1=\psi\circ \upsilon$.

Now let $t$ be a projective parameter on $J_\gamma^{(4-n)}$. E. Wilczynski
showed that for any $i$, $1\leq i\leq 2m-2$, the following degree
$i+2$ differentials
\begin{equation}
\label{willy} {\mathcal
W}_i(t)\stackrel{def}{=}\frac{(i+1)!}{(2i+2)!}\left(\sum_{j=1}^{i}(-1)^{j-1}
\frac{
(2i-j+3)!(2m-i+j-3)!}{(i+2-j)!j!}B_{2m-3-i+j}^{(j-1)}(t)\right)(dt)^{i+2}
\end{equation}
 on
$J_\gamma^{(4-n)}$ does not depend on the choice of the projective parameter. In other words, for any $\lambda\in \gamma$, $\mathcal W_i$ is the well defined homogeneous polynomial of degree $i+2$ on the tangent line to $J_\gamma(\lambda)$ or, equivalently, on the tangent line to the abnormal extremal $\gamma$ at $\lambda$ 
The form ${\mathcal W}_i$ is called the \emph {$(i+2)$-th order Wilczynski
invariant of the curve $J_\gamma^{(4-n)}$}.

\begin{remark}
\label{ratnormrem}
Among all regular curves in the projective space $\mathbb P^k$ of dimension $k$,
all curves
with all  Wilczynski invariants equal to zero, belong to the \emph{rational normal curve}, i.e. to the curve
consisting  of the points of the form  $[t^k:t^{k-1}s\ldots:ts^{k-1}:s^k]$ in some homogeneous coordinates. $\Box$
\end{remark}

Note that the curve $J_\gamma^{(4-n)}$ is not an arbitrary regular curve in the projective space $\mathbb P W$. It satisfies the following additional property:
\medskip

{\bf (S1)} \emph{ The $(n-4)$th-osculating space of $J_\gamma^{(4-n)}$  at any point $\lambda$ 
is Lagrangian. }
\medskip

As shown already by Wilczynski \cite{wilch} such curves are \emph{self-dual} in the following sense:

{\bf (S2})
\emph{The curve $(J^{(n-4)}_\gamma)^*$ in the projectivization $\mathbb P W_\gamma^*$ of the dual space $W_\gamma^*$, which is dual to the curve of hyperplanes  $J^{(n-4)}_\gamma$ obtained from the original curve  $J_\gamma^{(4-n)}$ by the osculation of order $2(n-4)$,  is  equivalent to the original curve $J_\gamma^{(4-n)}$, i.e. there is a linear transformation $A:W\mapsto W^*$ sending $J^{(n-4)}_\gamma$ onto $(J^{(n-4)}_\gamma)^*$.}

Note that in contrast to property (S1) the formulation of property (S2) does not involve a symplectic structure on $W_\gamma$. However,  it can be shown \cite{wilch, kwessi} that if the property (S2) holds then there exists a unique, up to a multiplication by a nonzero constant, symplectic structure on $W_\gamma$ such that the property (S1) holds (here it is important that $\dim\, W_\gamma$ is even; similar statement for the case of odd dimensional linear space involves nondegenerate symmetric forms instead of skew-symmetric ones). Since in our case the symplectic structure on $W_\gamma$ is a priori given, in the sequel we will consider projective spaces of linear symplectic spaces only and by self-dual curves we will mean curves satisfying property (S1).

It was shown by Wilczynski that a curve in a projective space is self-dual if an only if all Wilczynski invariant of odd order vanish (for the modern Lie-algebraic interpretation of this fact see \cite{flag1}).
The remaining $n-4$ Wilczynski invariants of even order, $\mathcal W_{2i}$, $1\leq i\leq n-4$, constitute the fundamental set of symplectic invariants of the unparametrized curve $J_\gamma^{(4-n)}$. Note that the nonzero Wilczynski invariants have order $\geq 4$ in this case and $\mathcal W_2=B_{2m-4}(t)(dt)^4$ in any projective parameter $t$.

Taking the Wilczynski invariants $\mathcal W_{2i}$ for the Jacobi curves of all abnormal extremals living in the set $\mathcal R_D$ we obtain the invariants of the distribution $D$, called the \emph {generalized Wilczynski invariants of $D$ of order $2(i+1)$} and denoted also by $\mathcal W_{2i}$.
By the constructions, the generalized Wilczynski invariant  $\mathcal W_{2i}$ at a regular point $\lambda$ of $(D^2)^\perp\backslash (D^3)^\perp$,
is a special homogeneous polynomial of degree $2(i+1)$ on the tangent line  at $\lambda$ to the abnormal extremal passing through $\lambda$. Another interpretation of these invariants, as certain functions on fibers of $(D^2)^\perp$, defined by a multiplication on a constant on a fiber,  is given in \cite{zelvar}.

Note that in the case $n=5$ the only possibly nonzero generalized Wilczynski invariant is $\mathcal W_2$ and it has degree $4$. As shown in \cite{zelcart}, under an appropriate identification, this invariant coincides with the classical Cartan invariant obtained in \cite{cart10}.

%
%

\section{Canonical frames for rank 2 distributions of maximal class}
\label{londonsec}
\label{canfrsec}
\setcounter{equation}{0}

Now let $\mathcal R_D$ be the set if all regular points of $(D^2)^\perp\backslash(D^3)^\perp$.
Denote by ${\mathfrak P}_\lambda$ the set of all
projective parameterizations $\psi$ on the characteristic curve $\gamma$ , passing through
$\lambda$, such that $\psi(0)=\lambda$. Let
$$\Sigma_D=\{(\lambda, \psi):\lambda\in {\mathcal R}_D,
\psi\in {\mathfrak P}_\lambda\}.$$ Actually, $\Sigma_D$
is a principal bundle over ${\mathcal R}_D$ with the
structural group of all M\"{o}bius transformations,
preserving $0$ and $\dim \, \Sigma_D=2n-1$. The main results of \cite{doubzel1,doubzel2} can be summarized in the following:


\begin{theorem}
\label{frtheor}
For any rank $2$ distribution in $\mathbb R^n$ with  $n>5$ of maximal class
there exists the canonical, up to the action of $\mathbb Z_2$, frame on the corresponding
$(2n-1)$-dimensional manifold $\Sigma_D$ so that two distributions from the considered class are equivalent if and only if their canonical frames are equivalent. The group of symmetries of such
distributions is at most $(2n-1)$-dimensional and this upper bound is sharp.
All distributions from the considered class with
$(2n-1)$-dimensional Lie algebra of infinitesimal
symmetries is locally equivalent to the distribution $D_{((0,\ldots 0)}$ generated by the vector fields $X_1$ and $X_2$ from \eqref{X1}-\eqref{X2} with all $r_i$ equal to $0$
or, equivalently, associated
with the underdetermined ODE $z'(x)=\bigl(y^{(n-3)}(x)\bigr)^2$.
The
symmetry algebra of this distribution is isomorphic to a semidirect
sum of $\mathfrak{gl}(2,\mathbb R)$ and $(2n-5)$-dimensional
Heisenberg algebra ${\mathfrak n}_{2n-5}$ such that $\mathfrak{gl}(2,\mathbb R)$  acts irreducibly on a complement of the center of  ${\mathfrak n}_{2n-5}$ to ${\mathfrak n}_{2n-5}$ itself .
\end{theorem}

\section{Canonical frames for rank 2 distributions of maximal class with distinguish parametrization on abnormal extremals}
\label{canfrparsec}
\setcounter{equation}{0}

Now assume that at least one generalized Wilczynski invariant of rank $2$ distribution $D$ of maximal class does not vanish. Let $i_0$  be the minimal integer such
 that $\mathcal W_{2i_0}$ on some (open) subset $\widetilde{\mathcal R}$ of $\mathcal R_D$. Then on each segment  $\gamma$ of abnormal extremals lying in $U$ we can choose the unique, up to a shift and the change of the orientation,  parametrization $\varphi:I\mapsto \gamma$ such that
 \begin{equation}
 \label{canparam}
 \mathcal W_{2i_0}\bigl(\psi(t)\bigl)\bigl(\psi^\prime(t)\bigr)\equiv \varepsilon,
 \end{equation}
 where $\varepsilon=1$ if $\mathcal W_{2i_0}>0$  and  $\varepsilon=-1$ if $\mathcal W_{2i_0}<0$.
 Here by $\mathcal W_{2i_0}(\lambda)(v)$, where $\lambda\in U$ and $v$ is the tangent vector at $\lambda$ to the abnormal extremal  $\gamma$ passing to $\lambda$, we mean the generalized Wilczynski invariant $\mathcal W_{2i_0}$ at $\lambda$ evaluated at $v$.

 Moreover, we also can fix the orientation on the curve. For this note that since the curve $J_\gamma^{(4-n)}$ is self-dual,  given a parametrization $\psi$ on $\gamma$, among all canonical sections of the tautological bundle $C(J_\gamma^{(4-n)})$ (defined up to the
multiplication by a nonzero constant)
there exists the unique, up to a sign, section $E$ of  such that \eqref{last0} holds and
\begin{equation}
\label{norme1}
\left|\tilde\sigma_\gamma\left(\frac{d^{n-3}}{dt^{n-3}}E\bigl(\psi(t)\bigr),\frac{d^{n-4}}{dt^{n-4}}
E\bigl(\psi(t)\bigr)\right)\right|\equiv 1.
\end{equation}
This section $E$ will be called the \emph{ strongly canonical section} of $C(J_\gamma^{(4-n)})$ with respect to the parametrization $\psi$.
The parametrization $\psi$ is called the \emph{canonical parametrization of the abnormal extremal $\gamma$} if \eqref{canparam} holds and
\begin{equation}
\label{norme1orient}
\tilde\sigma_\gamma\left(\frac{d^{n-3}}{dt^{n-3}}E\bigl(\psi(t)\bigr),\frac{d^{n-4}}{dt^{n-4}}
E\bigl(\psi(t)\bigr)\right)\equiv 1.
\end{equation}

 Another situation when a special parametrization, up to  a shift, can be distinguished is the case of regular control systems on rank $2$ distributions in the sense Definition \ref{regcontdef}.
 Let  $\mathcal V_q$ be the set of admissible velocities of the control system under consideration at the point $q \in M$.
 Let $\widehat {\mathcal R}$ be  a subset of $\mathcal R_D$ consisting of all points $\lambda$ such that the image under $d\pi$ of the tangent line at $\lambda$ to the abnormal extremal passing through $\lambda$ is a regular line in $D\bigl(\pi(\lambda)\bigr)$ in the sense of Definition \ref{reglinedef} (here , as before  $\pi:T^*M\rightarrow M$ is the canonical projection).
 Then by Definition \ref{regcontdef} and Remark \ref{openrem} the set $\widehat {\mathcal R}$ is a non-empty open subset of $(D^2)^\perp$.
 Given a regular line $L$ in $D(q)$ let $w(L)$ be the admissible velocity in $L$ of the smallest norm. Clearly $w(L)$ does not depend on the choice of a norm in $D(q)$, but in general it may be defined up to a sign (for example, in the sub-(pseudo) Riemannian case).

 A parametrization $\psi:I\mapsto \gamma$ of an abnormal extremal $\gamma$ living in $\widehat {\mathcal R}$ is called \emph{ weakly canonical (with respect to the regular control system given by the set of admissible velocities $\{\mathcal V_q\}_{q\in M}$)}  if
 \begin{equation}
 \label{canparaff}
 d\pi\bigl(\frac{d}{dt}\gamma(\psi(t))\bigr)=w\Bigl({\rm   span}\,d\pi\bigl(\frac{d}{dt}\gamma(\psi(t))\bigr)\Bigr)
 \end{equation}
This parametrization is defined up to a shift and maybe up to the change of orientation. In the case when the orientation is not fixed by \eqref{canparaff} we can fix it by imposing
the condition \eqref{norme1orient}. In any case we finally obtain the parametrization of $\gamma$ defined up to a shift only. This parametrization of $\gamma$ is called \emph{canonical (with respect to the regular control system given by the set of admissible velocities $\{\mathcal V_q\}_{q\in M}$)}.

Finally let $\widetilde {\mathcal R}$ be a subset of $\widehat {\mathcal R}$ where the vector field consisting of the tangent vectors to the abnormal extremals parameterized by the canonical parameter is smooth.
Note that $\widetilde {\mathcal R}$ is an open and dense subset of $\widehat {\mathcal R}$. For affine control systems with one input and a non-zero drift and for sub-Riemannian structures $\widehat {\mathcal R}$ coincides with the set $\mathcal R_D$ of the regular points in $(D^2)^\perp\backslash(D^3)^\perp$.


Note that in all cases the canonical parametrization is preserved by the homotheties of the fibers of $(D^2)^\perp$. Namely,
if $\delta_s$ is the flow of homotheties on
 the fibers of $T^*M$:
$\delta_s(p,q)=(e^sp,q),\quad q\in M,\,\,
p\in T_q^*M$ or, equivalently, the flow generated by the Euler field $e$ generates this flow, then
$\psi:I\mapsto \gamma$ is the canonical parametrization on an abnormal extremal $\gamma$ if and only if
$\delta_{s}\circ\psi$ is the canonical parametrization on the abnormal extremal $\delta_s\circ\gamma$.


The main goal of this section is to prove the following
 \begin{theorem}
\label{theor1par}
Given either a rank 2 distribution $D$ of maximal class with at least one nonvanishing  generalized Wilczynski invariant
or a regular control system on a rank 2 distribution $D$ of maximal class (even with the identically vanishing Wilczynski invariants), one can assign to such structure a canonical, up to the action of $\mathbb Z_2$, frame  on the set $\widetilde {\mathcal R}$ defined above so that two objects from the considered class are equivalent if and only if their canonical frames are equivalent.
\end{theorem}

\begin{proof}
First, let $h$ be the vector field consisting of the tangent vectors to the abnormal extremals parameterized by the canonical parameter.

Second, given $\lambda \in (D^2)^\perp$ denote by $V(\lambda)$ the tangent space to the fiber of the bundle $\pi\colon (D^2)^\perp\mapsto M$
(the vertical subspace of $T_\lambda (D^2)^\perp$),
\begin{equation}
\label{verteq}
V(\lambda)= \{v\in T_\lambda
(D^2)^\perp,\pi_*v=0\}.
\end{equation}
It is easy to show (\cite{doubzel2, zelvar}) that
\begin{equation}
\label{J-1}
d\phi \bigl(V(\lambda)\oplus \mathcal C(\lambda)\bigr)=J_\gamma^{(-1)}(\lambda)\quad {\rm mod}\,\mathbb R\bar e,
\end{equation}
where $\phi$ is as in \eqref{phi}, $\bar e=\phi_* e$ with $e$ being the Euler field, and $\gamma$ is the abnormal extremal passing through $\lambda$. Define also the following subspaces of $T_\lambda (D^2)^\perp$:

\begin{equation}
\label{calJi}
\mathcal J^{(i)}(\lambda)=\{w\in T_\lambda (D^2)^\perp: d\phi(w)\in J_\gamma^{(i)}(\lambda)\,\,{\rm mod}\,\mathbb R\bar e\}.
\end{equation}

Directly from the definition, if $\lambda \in \mathcal R_D$, then
\begin{equation}
\label{adC}
[\mathcal C, \mathcal J^{(i)}](\lambda)=\mathcal J^{(i+1)}(\lambda).
\end{equation}
Also, if
$V^{(i)}(\lambda)=V(\lambda)\cap\mathcal J^{(i)}(\lambda)$, then
\begin{equation}
\label{Jisplit}
\mathcal J^{(i)}(\lambda)=V^{(i)}(\lambda)\oplus\mathcal C(\lambda) \quad \forall i\leq 0.
\end{equation}

Moreover, it can be shown (\cite[Lemma 2]{doubzel2}) that
\begin{equation}
\label{VV}
[V^{(i)}, V^{(i)}]\subseteq V^{(i)},\quad [V^{(i)},\mathcal J^{(i)}]\subseteq \mathcal J^{(i)}, \quad \forall i\leq 0.
\end{equation}

Let, $E$ be the strongly canonical section of $C(J_\gamma^{(4-n)})$ with respect to the canonical parametrization $\psi$ of the abnormal extremal $\gamma$ (as defined by \eqref{norme1}).
Then \eqref{J-1} implies that a vector field $\epsilon_1$ such that

\begin{enumerate}
\item[(A1)]
$d\phi\bigl(\varepsilon_1(\lambda)\bigr)\equiv E \,\,\,{\rm mod}\,\bar e$,
\item[(A2)]
$\varepsilon_1$ is the section of the vertical distribution $V$
\end{enumerate}
is defined modulo the Euler field $e$. Note that conditions (A1) and (A2) also imply that $\varepsilon_1$ is the section of $V^{(4-n)}$.

\begin{lemma}
\label{normlem}
Among all vector fields $\varepsilon_1$ satisfying conditions (A1) and (A2), there exists the unique, up to a multiplication by $-1$, vector field such that
\begin{equation}
\label{normcond}
\bigl[\varepsilon_1, [h, \varepsilon_1]\bigr](\lambda)\in{\rm span}\{e(\lambda),h(\lambda),\varepsilon_1(\lambda)\}.
\end{equation}
\end{lemma}
\begin{proof}
Let $\tilde\varepsilon_1$ be a vector field satisfying the conditions (A1) and (A2). Then $\tilde\varepsilon_1$ is the section of $V^{(4-n)}$. Using  \eqref{Jisplit} and \eqref{VV} for $n>5$ and also the definition of $\mathcal J$ given by \eqref{prejac} in the case $n=5$, we get
\begin{equation}
\label{prenorm}
\bigl[\tilde\varepsilon_1, [h, \tilde\varepsilon_1]\bigr]\equiv k[h,\tilde\varepsilon_1]\,\,\rm{mod}\,{\rm span}\{e,h,\tilde\varepsilon_1\}
\end{equation}
for some function $k$.
Now let $\varepsilon_1$ be another vector field satisfying conditions (A1) and (A2). Then by above there exists a function $\mu$
such that
\begin{equation}
\label{transe1}
\varepsilon_1=\pm \tilde\varepsilon_1+\mu e.
\end{equation}

From the fact that the canonical parametrization is preserved by the homotheties of the fibers of $(D^2)^\perp$ it follows that $[e,h]=0$ . Also from the normalization condition \eqref{norme1} it is easy to get that
\begin{equation}
\label{ee1}
[e,\varepsilon_1]=-\cfrac{1}{2} \varepsilon_1 \,\,{\rm mod}\,{\rm span}(e).
\end{equation}
Then
 \begin{equation}
\label{ee1}
\bigl[e,[h,\varepsilon_1]\bigr]=-\cfrac{1}{2} [h,\varepsilon_1] \,\,{\rm mod}(e,h),
\end{equation}
From this and \eqref{transe1} it follows that
\begin{equation}
\label{prenorm1}
\bigl[\varepsilon_1, [h, \varepsilon_1]\bigr]\equiv \bigl(k\mp \cfrac{\mu}{2}\bigr)[h,\varepsilon_1]\,\,{\rm span}\{e,h,\varepsilon_1\},
\end{equation}
which implies the statement of the lemma: the required vector $\tilde\varepsilon_1$  is obtained by taking $\mu=\pm 2k$. $\Box$
\end{proof}

Now we are ready to construct the canonical frame on the set $\widetilde{\mathcal R}$. One option is to take as a canonical frame the following one:
\begin{equation}
\label{canframe1}
\bigl\{e,h,\varepsilon_1,\{({\rm ad}\, h)^i\varepsilon_1\}_{i=1}^{2n-7},[\varepsilon_1,({\rm ad}\, h)^{2n-7}\varepsilon_1]\bigr\},
\end{equation}
where $\varepsilon_1$ is as in Lemma \ref{normlem}. Let us explain why it is indeed a frame. First the vector fields $\bigl\{e,h,\varepsilon_1,\{({\rm ad}\, h)^i\varepsilon_1\}_{i=1}^{2n-7}\bigr\}$ are linearly independent on $\widetilde{\mathcal R}$ due to the relation \eqref{adC}. Besides $[\varepsilon_1,({\rm ad}\, h)^{2n-7}\varepsilon_1](\lambda)\notin \mathcal J^{(n-3)}(\lambda)$. Otherwise, $\varepsilon_1(\lambda)$ belongs to the kernel of the form $\sigma(\lambda)|_{(D^2)^\perp}$ and therefore it must be collinear to $h$. We get a contradiction. Therefore the tuple of vectors in \eqref{canframe1} constitute a frame on $\widetilde{\mathcal R}$.

The construction of the frame \eqref{canframe1} is intrinsic.
However, in order to guaranty that two objects from the considered class are equivalent if and only if their canonical frames are equivalent, we have to modify this frame such that it will contain the basis of the vertical distribution $V$ (defined by \ref{verteq}). For this, replace the vector fields of the form $(ad h)^i\varepsilon_1$ for $1\leq i\leq n-4$ by their projections  to $V^{(i)}$ with respect to the splitting \eqref{Jisplit}, i.e. their vertical components with respect to this splitting. This completes the construction of the required canonical frame (defined up to the action of the required finite groups). The proof of Theorem \ref{theor1par} is completed.
\end{proof}

As a direct consequence of Theorem \ref{theor1par} we have
\begin{corollary}
\label{cor1}
For a rank 2 distribution $D$ of maximal class with at least one nonvanishing  generalized Wilczynski invariant
or a regular control system on a rank 2 distribution $D$ of maximal class (even with the vanishing Wilczynski invariants) the dimension of pseudo-group of local symmetries does not exceed $2n-3$.
\end{corollary}

\section{Symplectic curvatures for the structures under consideration}
\label{sympsec}
\setcounter{equation}{0}

Before proving Theorems \ref{symmodthm1} and \ref{symmodthm0} about the most symmetric models for geometric structures under consideration,
we want to reformulate this theorem in more geometric  terms. For this we distinguish special invariants for this structures, called the \emph{symplectic curvatures}. In contrast to the generalized Wilczynski invariants they are functions on the open subset $\widetilde {\mathcal R}$ of  $\mathcal R_D$, defined in the beginning of the previous section.

\subsection{The case of regular control systems}
\label{affsbsec}

In this case all curves $J_\gamma^{(4-n)}$ are parameterized by the canonical (up to a shift) parametrization $\psi$ given by \eqref{canparaff} (and maybe also by \eqref{norme1orient}) .
The geometry of parameterized regular self-dual curves in projective spaces  is simpler than of unparametrized ones: instead of forms on the curve we obtain invariant, which are scalar-valued  function on the curve (\cite{zelrank1}). The main result of \cite{zelrank1} (Theorem 2 there) can be reformulated as follows (see also \cite{kwessi}: if $E$ is a (strongly) canonical section of $C(J_\gamma^{(4-n)})$ with respect to the (canonical) parametrization $\psi$, then there exist $m$ functions $\rho_1(t),\ldots,\rho_m(t)$ such that
\begin{equation}
\label{selfeq}
E^{(2m)}\bigl(\psi(t)\bigr)=\sum_{i=1}^{m}(-1)^{i+1}\cfrac{d^{m-i}}{dt^{m-i}}\Bigl(\rho_{i}(t) \cfrac{d^{m-i}}{dt^{m-i}}E\bigl(\psi(t)\bigr)\Bigr).
\end{equation}
Note that formula \eqref{selfeq} resembles the classical normal form for the formally self-adjoint linear differential operators \cite{naimark}[\S 1].

By constructions, the functions $\rho_1(t),\ldots,\rho_m(t)$ are invariants of the parameterized curve $t\mapsto J_\gamma^{(4-n)}\bigl(\psi(t)\bigl)$ with respect to the action of the linear symplectic group on $W_\gamma$. We call the function $\rho_i(t)$ the \emph{$i$th symplectic curvature of the parametrized curve $t\mapsto J_\gamma^{(4-n)}\bigl(\psi(t)\bigr)$.}
Besides, the functions $\rho_1(t),\ldots,\rho_m(t)$ constitute the fundamental system of symplectic invariant of the parametrized curve $t\mapsto J_\gamma^{(4-n)}\bigl(\psi(t)\bigr)$, i.e. they determine this curve uniquely up to a symplectic transformation. Moreover, these invariants are independent: for any tuple of $m$ functions $\rho_1(t),\ldots,\rho_m(t)$ on the interval $I\subseteq R$ there exists a parameterized regular self-dual curve $t\mapsto \Lambda(t)$, $t\in I$, in the projective space of dimension $2m-1$ with the $i$th symplectic curvature equal to $\rho_i(t)$ for any $1\leq i\leq m$.

Also in the sequel we will need the following
\begin{remark}
\label{symprodrem}
Assume that $E$ is the strongly canonical section of $C(J_\gamma^{(4-n)})$ with respect to the  parametrization $\psi$.
Using the fact that the spaces ${\rm  span}\Bigl \{\cfrac{d^j}{dt^j}E\bigl(\psi(t)\bigr)\Bigr\}_{j=1}^m$ are Lagrangian and the condition \eqref{norme1}, it is easy to show
that $$\tilde\sigma_\gamma\Bigl(\cfrac{d^j}{dt^j}E\bigl(\psi(t)\bigr), \cfrac{d^i}{dt^i}E\bigl(\psi(t)\bigr)\Bigr)$$ are either identically equal to $0$, if $i+j< 2m-1$ or to $\pm 1$, if $i+j=2m-1$, or they are polynomial expressions (with universal constant coefficients) with respect to the symplectic curvatures
$\rho_1(t),\ldots,\rho_m(t)$ and their derivatives, if $i+j>2m$. $\Box$
\end{remark}

Taking the $i$th symplectic curvature for Jacobi curves (parameterized by the canonical parameter) of all abnormal extremals living in $\widetilde{\mathcal R}$,  we obtain the invariants of the regular control systems, called the \emph {$i$th symplectic curvature}
 and denoted also by $\rho_i$. The symplectic curvatures are scalar valued functions on the set $\widetilde{\mathcal R}$.

\subsection{The case of rank 2 distributions with nonzero generalized Wilczynski invariants}
\label{wilczsubsec}

Assume that at least one generalized Wilczynski invariant of rank $2$ distribution $D$ of maximal class vanishes nowhere. Let $i_0$  be the minimal integer such
 that $\mathcal W_{2i_0}$ on some (open) subset $\widetilde{\mathcal R}$ of $\mathcal R_D$.
 In this case given an abnormal extremal $\gamma$ in $\widetilde{\mathcal R}$ the curve $J_\gamma^{(4-n)}$ are parameterized by the canonical, up to a shift, parametrization $\psi$ given by \eqref{canparam} and \eqref{norme1orient}. Let $\rho_i$ be the $i$th symplectic curvatures of the parametrized curve
$t\mapsto J_\gamma^{(4-n)}\bigl(\psi(t)\bigl)$. Then $\rho_i$ is an invariant of the (unparametrized) curve $J_\gamma^{(4-n)}$ with respect to the action of the linear symplectic group on $W_\gamma$ and a reparametrization of $\gamma$, because , in contrast to the previous case the parametrization on $\gamma$ is not a priori prescribed but determined by the unparametrized curve $J_\gamma^{(4-n)}$ itself. We say that $\rho_i$ is the \emph{$i$th symplectic curvature of the unparametrized curve $J_\gamma^{(4-n)}$}. By Remark \ref{ratnormrem} this notion is defined for all regular curves in projective spaces except the rational normal curve.

\begin{remark}
\label{rho1}
Note also that by \ref{rhorep} the first symplectic curvature $\rho_1$  is equal , up to the universal constant multiple $\frac{-m(4m^2-1)}{3}$, to the Schwarzian derivative of the transition function from the canonical parametrization of $\gamma$   to any projective parametrization of $\gamma$. The invariant $\rho_1$ coincides here, up to the universal constant, with the projective Ricci curvature of the unparametrized curve in a Lagrangian Grassmannian introduced in \cite{zelvar, phd}.$\Box$
\end{remark}

In contrast to the case considered in subsection \ref{affsbsec} the invariant $\rho_1,\ldots,\rho_m$ are dependent so that there are $i_0$ polynomial relations (with universal coefficients) between them and their derivatives (with respect to the canonical parametrization).  Therefore the following definition makes sense.

\begin{definition}
\label{comptdef}
The tuple of functions $\Bigl(r_1(\tau),\ldots,r_m(\tau)\Bigr)$ is called \emph{compatible} if there exists a regular self-dual curve $\Lambda$ in the projective space of dimension $2m-1$ with at least one nowhere zero Wilczynski invariant such that if $\rho_i$ is the $i$th symplectic curvature of the unparametrized curve $\Lambda$, then for all $1\leq i \leq m$ we have that $r_i(\tau)=\rho_i\bigl(\psi(\tau)\bigr)$ for the canonical parametrization $\psi$ of $\Lambda$.
\end{definition}

To explain why the invariants $\Bigl(r_1(\tau),\ldots,r_m(\tau)\Bigr)$ are dependent note that there is another way to construct the scalar-valued invariant of the unparametrized curve $J_\gamma^{(4-n)}$ (up to a symplectic transformation and
reparametrization) with the help of the Wilczynski invariants. Namely, for any $1\leq i\leq m-1$ let
\begin{equation}
\label{Ai}
 A_i(\psi(t))=\mathcal W_{2i}(\psi(t))\bigl(\psi'(t)\bigr)
\end{equation}
Then by the definition of $i_0$ and the canonical parametrization, $A_i\equiv 0$ for $1\leq i\leq i_0-1$ and $A_{i_0}=\varepsilon$, where $\varepsilon=1$ if $\mathcal W_{2i_0}>0$  and  $\varepsilon=-1$ if $\mathcal W_{2i_0}<0$. Using the transformation rule \eqref{EE} with $\psi$ being a projective parametrization of $\gamma$ and $\psi_1$ being the canonical parametrization of $\gamma$ and also the Remark \ref{rho1}, it can be shown that all $\rho_i$ with $1\leq i\leq i_0$ can be expressed as a certain polynomial in $\rho_1$ and its derivatives (with respect to the canonical parameter) having universal coefficients and with a free term equal to zero if $1\leq i<i_0$ and equal to $(-1)^{i_0-1}\varepsilon$ if $i=i_0$. For example, if $i_0=1$, i.e. the nontrivial Wilczynski invariant $W_2$ of the lowest order is non-zero, then there is only one relation
\begin{equation}
\label{i0=1}
\rho_2\bigl(\psi_1(\tau)\bigr)=-\varepsilon-\frac{3(2m-2)(2m-3)}{20}\frac{d^2}{d\tau^2}\rho_1\bigl(\psi(\tau)\bigr)+\frac{\alpha_{m,2}}{\alpha_{m,1}^2}\rho_1\bigl(\psi(\tau)\bigr)^2,
\end{equation}
where the constants $\alpha_{m,1}$ and $\alpha_{m,2}$ are given by formula \ref{c_iprop}. Therefore a tuple $\Bigl(r_1(\tau),\ldots,r_m(\tau)\Bigr)$
satisfying  relation \eqref{i0=1}, with $\rho_i\bigl(\psi(\tau)\bigr)$ is replaced by $r_i(\tau)$, is compatible.
 Relation \eqref{i0=1} is analogous to \cite[Lemma 5.1]{jac1}. The first relation for $i_0>1$ is obtained from \eqref{i0=1} by replacing $\varepsilon$ with $0$.

 The deduction of other relations in an explicit form in general needs an extra work and we will not do it here, because we do not need such explicit relations in the sequel. However, in the case when all invariants $\rho_i(\tau)$ are constants, $\rho_i(t)\equiv r_i$, $1\leq i\leq m$, there is much more elegant way to explain the role of the coefficients of polynomial \eqref{c_iprop} in the question of compatibility of the tuple $(r_1,\ldots,r_m)$, based on some elementary facts from the representation theory of the Lie algebra $\mathfrak{sl}_2$.

 Given a tuple of $m$ numbers $(r_1,\ldots,r_m)$ let $\tau\mapsto\Lambda_{(r_1,\ldots,r_m)}(\tau)$ , $\tau\in R$ be the parameterized self-dual curve in the $(2m-1)$-dimensional projective space $\mathbb P^{2m-1}$ with the $i$th symplectic curvature constantly equal to
$r_i$ for all $1\leq i\leq m$. Note that the closure of the curve $\Lambda_{(0,\ldots,0)}$ is the rational normal curve.

 \begin{proposition}
\label{Wilczexceptrem}
 The curve $\Lambda_{(r_1,\ldots,r_m)}$ has all Wilczynski invariants equal to zero if and only if the tuple
$(r_1,\ldots,r_m)$ is exceptional in the sense of Definition \ref{exceptdef}.
\end{proposition}

\begin{proof}
Assume that the curve $\Lambda_{(r_1,\ldots,r_m)}$ has all Wilczynski invariants equal to zero. Then by Remark \ref{ratnormrem}  $\Lambda_{(r_1,\ldots,r_m)}$ belongs to a rational normal curve. It is well known \cite{wilch, doub3, flag2} that the algebra of infinitesimal symmetries of the rational normal curve (with respect to the action of $SL_{2m}$) is isomorphic to $\mathfrak{sl}_2$ and it is actually equal to the image of the irreducible
embedding of $\mathfrak{sl}_2$ into $\mathfrak{sl}_{2m}$.

To describe these infinitesimal symmetries note that $\mathbb R^{2m}$ can be identified with the projectivization of the space of homogeneous binary polynomials of degree $2m-1$ (say in variables $x_1$ and $x_2$), $\mathbb R^{2m}\cong {\rm Sym}^{2m-1}(\mathbb R^2)$.  With this identification, the standard action of the group $SL_2$ (the algebra $\mathfrak{sl}_2$) on $\mathbb R^2$ with coordinates $(x_1,x_2)$ induces the standard irreducible representation of  $SL_2$ ($\mathfrak{sl}_2$)  into $SL_{2m}$ (${\mathfrak sl}_{2m}$)
 The rational normal curve $\bar \Lambda_{(0,\ldots,0)}$, up to a projective transformation, is the projectivization of the binary polynomials which are the $(2m-1)$th power of the linear forms in $x_1$ and $x_2$ . Therefore  any element of the image of the standard irreducible representation of $SL_2$ in $SL_{2m}$ preserves the curve $\bar \Lambda_{(0,\ldots,0)}$ and any element of the image of the standard irreducible representation of $\mathfrak{sl}_2$ in $\mathfrak{sl}_{2m}$ defines an infinitesimal symmetry of this curve. Moreover, it can be shown that there are no other infinitesimal symmetries of this curve.

By the standard theory of $\mathfrak{sl}_2$-representations (\cite[\S 11.1]{fult}) the image under this representation of any diagonizable element of $\mathfrak {sl}_2$ has the spectrum of the form
\begin{equation}
\label{arithm}
\{-(2m-1)r,-(2m-3)r,\ldots, -r, r,\ldots, (2m-3)r, (2m-1)r\}.
\end{equation}
(where over $\mathbb R$ the number $r$ is either real or purely imaginary), which is an arithmetic progression symmetric with respect to the origin. Besides any element of $\mathfrak{sl}_2$ can be brought to the triangular form (over $\mathbb C$), therefore its image under the aforementioned embedding has also the spectrum of the form \eqref{arithm}. In other words, \emph{the spectrum of any infinitesimal symmetry of $\Lambda_{(r_1,\ldots,r_m)}$ is an arithmetic progression symmetric with respect to the origin.}.

On the other hand, since all symplectic invariants of the curve $\tau\mapsto\Lambda_{(r_1,\ldots,r_m)}(\tau)$ are constants, this curve belongs to the orbit of the one-parametric group generated by an element $X_{(r_1,\ldots,r_m)}$ of the symplectic algebra (for the explicit form of $X_{r_1,\ldots,r_m}$
see \cite{zelrank1}, where it is exactly the matrix in the structure equation for the canonical moving frame of the curve $\tau\mapsto\Lambda_{(r_1,\ldots,r_m)}(\tau)$).
Therefore $X_{(r_1,\ldots,r_m)}$ belongs to the
algebra of infinitesimal symmetries  of the curve $\Lambda_{(r_1,\ldots,r_m)}$. Hence, by above its spectrum is an arithmetic progression  symmetric with respect to the origin. Finally, from the explicit form of $X_{(r_1,\ldots,r_m)}$ given in \cite{zelrank1} it follows that the characteristic polynomial of $X_{(r_1,\ldots,r_m)}$ is exactly the polynomial \eqref{charpoly}, which completes the proof of one direction of the proposition.

In opposite direction, assume that the tuple $(r_1,\ldots,r_m)$ is exceptional in the sense of Definition  \ref{exceptdef}. Then the corresponding element $X_{(r_1,\ldots,r_m)}$  has the matrix $S+N$ in some basis, where $S$ is the diagonal matrix with the entries on the diagonal as in \eqref{arithm} for some $r$ (and in the same order) and $N$ is the Jordan block. Then from the assumptions on the spectrum of $X_{(r_1,\ldots,r_m)}$ it follows that in this basis $X_{(r_1,\ldots,r_m)}$ can be considered as an element of the image of the standard irreducible embedding of $\mathfrak{sl}_2$ into $\mathfrak{sl}_{2m}$. This embedding is the algebra of infinitesimal symmetries of the orbit of the first coordinate line with respect to the one-parametric group generated by $N$. Consequently, our curve $\Lambda_{(r_1,\ldots,r_m)}$ belongs to the closure of this orbit, which in turn is a rational normal curve. This completes the proof of our proposition.
$\Box$
\end{proof}



 From Proposition \ref{Wilczexceptrem} it follows that an exceptional tuple $(r_1,\ldots,r_m)$ in the sense of Definition \ref{exceptdef} is not compatible. Another consequence is the following

 \begin{corollary}
 \label{D0cor}
 The distribution $D_{(r_1,\ldots,r_{n-3})}$ is locally equivalent to $D_{(0,\ldots,0)}$ (or, equivalently, has the algebra of infinitesimal symmetries of the maximal possible dimension among all rank $2$ distributions of maximal class in $\mathbb R^n$) if and only if the tuple
$(r_1,\ldots,r_{n-3})$ is exceptional in the sense of Definition \ref{exceptdef}.
 \end{corollary}
\begin{proof}
 In \cite{var} we established that the following three equivalence problems are the same after an appropriate identification of the objects involved in them: the equivalence of rank $2$ distributions of a special type, namely, associated with underdetermined ODE (the Monge equation) $z'(x)=F\bigl(x,y(x),\ldots, y^{(n-3)}(x)\bigr)$, the equivalence of the Lagrangians as in \eqref{Lagr}, up to a contact transformation, a multiplication by a nonzero constant, and modulo divergence, and the equivalence of their Euler-Lagrange equations, up to a contact transformation. Moreover, the latter problem, in the case when the Euler-Lagrangian equation is linear coincides with the equivalence of the corresponding self-dual curves in projective spaces (up to a linear symplectic transformation).
 The distribution $D_{(r_1,\ldots,r_{n-3})}$ is associated with the Monge equation \eqref{under}, which in turn corresponds to the Lagrangian \eqref{Lagr} 
 having the linear Euler-Lagrange equation. So, the question of equivalence of the distributions $D_{(r_1,\ldots,r_{n-3})}$ and $D_{(0,\ldots,0)}$ is reduced to the question of the equivalence of the curves $\Lambda_{(r_1,\ldots,r_m)}$ and $\Lambda_{(0,\ldots,0)}$ from Proposition \ref{Wilczexceptrem}. Hence, our Corollary follows from  Proposition \ref{Wilczexceptrem}.
 $\Box$
 \end{proof}

The following simple lemma will be useful in the next section

\begin{lemma}
\label{comptlem}
 Let  $(r_1,r_2\ldots,r_m)$ be a tuple of $m$ constants which is not exceptional in the sense of Definition \ref{exceptdef}. Then among all tuple of the form $(c^2 r_1, c^4 r_2,\ldots, c^{2m}r_m)$, where $c$ is an arbitrary non-zero constant, there exists exactly one tuple which is compatible in the sense of Definition \ref{comptdef}
\end{lemma}

\begin{proof}
Let the curve $\tau\mapsto \Lambda_{(r_1,r_2\ldots,r_m)}(\tau)$ be as in Proposition \ref{Wilczexceptrem}.
Then since all $r_i$ are constants, the curves $\Lambda(\tau)$ and $\Lambda (\tau-a)$ are equivalent, up to a symplectic transformation. Therefore the Wilczynski invariants have the form $\mathcal W_{2i}=A_i\,(dt)^i$, $1\leq i\leq m-1$, where all $A_i$ are constants. Besides, not all of $A_i$ are zero, otherwise by Proposition \ref{Wilczexceptrem} the tuple $(r_1,r_2\ldots,r_m)$ is exceptional. Hence, there is a constant $c$ such that $\tau=ct$ is the canonical parameter on the curve $\Lambda$. Then the statement of the lemma follows from the formula \eqref{selfeq}. $\Box$
\end{proof}

\begin{remark}
In the case when $\rho_1$ is constant it can be shown that
\begin{eqnarray}
\label{rho1constant}
&~&\rho_i\equiv\frac{\alpha_{m,i}}{\alpha_{m,1}^{i}}(\rho_1)^{i}, \quad\forall 1\leq i\leq i_0-1, \label{rho1constant1}\\
&~&\rho_{i_0}\equiv (-1)^{i_0-1}\varepsilon+\frac{\alpha_{m,i_0}}{\alpha_{m,1}^{i_0}}\rho_1^{i_0}\label{rho1constant2}.
\end{eqnarray}
where the constants $\alpha_{m,i}$ are given by formula \ref{c_iprop}.
If $\rho_1$ is constant and all $\rho_i$, with $2\leq i\leq i_0$ satisfy relations \eqref{rho1constant1}-\eqref{rho1constant2}, then the tuple $\bigl(\rho_1,\ldots,\rho_{i_0},\rho_{i_0+1}(\tau),\ldots,\rho_{m}(\tau)\bigr)$ is compatible.
Another way to prove Lemma \ref{comptlem} is by using this description of the compatible tuples: the required constant $c$ can be found explicitly using formula \eqref{rho1constant2}. $\Box$
\end{remark}

Finally, taking the $i$th symplectic curvature for Jacobi curves (parameterized by the canonical parameter) of all abnormal extremals living in $\widetilde{\mathcal R}$,  we obtain the invariants of the rank $2$ distribution $D$, called the its \emph {$i$th symplectic curvature } and denoted also by $\rho_i$. The symplectic curvatures are scalar valued functions on $\widetilde{\mathcal R}$.
\section{The maximally symmetric models}
\label{symmodsec}
\setcounter{equation}{0}

Now we will find all structures from the considered classes having the pseudo-group of local symmetries of dimension equal to $2n-3$.
As  a consequence of Corollary \ref{cor1} if an object from the considered class has the pseudo-group of local symmetries of dimension equal to $2n-3$
then all structure functions of the canonical frame \eqref{canframe1} must be constant. Note  that formula \eqref{selfeq} can be rewritten in terms of the canonical frame  \eqref{canframe1} as follows
\begin{equation}
\label{adselfeq}
[h, \varepsilon_{2m}]=\sum_{i=1}^{m}(-1)^{i+1}({\rm ad}\, h)^{m-i}\Bigl(\rho_{i}\bigl({\rm ad}\, h^{m-i}\varepsilon_1\Bigr) \quad\rm{mod}\,\,{\rm span}\{e, h\},
\end{equation}
where $\rho_i$ are the $i$th symplectic curvatures of a structures under consideration.
 This implies that the symplectic curvatures of all order must be constant for any structure from the considered classes having $2n-3$-dimensional pseudo-group of local symmetries . This together with Corollary \ref{D0cor} and Lemma \ref{comptlem} for the case of distributions implies that the following two theorems are equivalent to Theorems \ref{symmodthm0} and \ref{symmodthm1}, respectively

\begin{theorem}
\label{symmodthm90}
Given any tuples of $n-3$ numbers $ (r_1,\ldots,r_{n-3})$ there exists the unique, up to local equivalence, regular control system
on a rank 2 distribution of maximal class in $\mathbb R^n$ with $n\geq 5$ having the  group of local symmetries of dimension $2n-3$ and the $i$th  symplectic curvature identically equal to $r_i$ for any $1\leq i\leq n-3$. Such regular control system is affine and locally equivalent to the system $A_{(r_1,\ldots,r_{n-3})}$ defined by
\eqref{affinemax1}-\eqref{X2}.
\end{theorem}

\begin{theorem}
\label{symmodthm91}
Given any tuples of $n-3$ numbers $ (r_1,\ldots,r_{n-3})$ compatible in the sense of Definition \ref{comptdef} there exists the unique, up to local equivalence,   rank 2 distribution in $R^n$ of maximal class, $n\geq 5$, with at least one nowhere vanishing generalized Wilczynski invariant such that its group of local symmetries has  dimension $2n-3$
  and the $i$th symplectic curvature
  is identically equal to $r_i$ for any $1\leq i\leq n-3$. Such distribution is locally equivalent to the distribution $D_{(r_1,\ldots,r_{n-3})}$ spanned by the vector fields from
\eqref{X1}-\eqref{X2}.
\end{theorem}

\begin{proof} We prove Theorems \ref{symmodthm90} and \ref{symmodthm91} simultaneously.

Let us prove the uniqueness. Take a structure from the considered class  having the  pseudo-group of local symmetries of dimension $2n-3$ and the $i$th  symplectic curvature identically equal to $r_i$ for any $1\leq i\leq m$, where, as before, $m=n-3$. Then, as was already mentioned, all structure functions of the canonical frame \eqref{canframe1} must be constant. The uniqueness will be proved if we will show that all nontrivial structure function (i.e. those that are not prescribed by the normalization conditions for the canonical frame) are uniquely determined by the tuple $ (r_1,\ldots,r_{n-3})$.

Let $\varepsilon_1$ be as in the Lemma \ref{normlem}. Denote
\begin{equation}
\label{ei}
\varepsilon_{i+1}:=({\rm ad}\, h)^i\varepsilon_1,\quad \nu=[\varepsilon_1,\varepsilon_{2m}]
\end{equation}
In this notations the canonical frame \eqref{canframe1} is $\{e,h,\varepsilon_1,\ldots,\varepsilon_{2m}, \eta\}$.
\begin{enumerate}
\item Let us prove that
\begin{equation}
\label{ee1const}
[e,\varepsilon_1]=-\frac{1}{2} \varepsilon_1
\end{equation}
where, as before $e$ is the Euler field. Indeed, from \eqref{ee1}
\begin{equation}
\label{ee1prel}
[e,\varepsilon_1]=-\frac{1}{2} \varepsilon_1+a e
\end{equation}
where $a$ is constant by our assumptions. Then, using the Jacobi identity and the fact that
\begin{equation}
\label{eh}
[e,h]=0
\end{equation}
we get that
\begin{equation}
\label{ee2const}
[e,\varepsilon_2]=\bigr[e,[h,\varepsilon_1]\bigl]=\bigl[h,[e,\varepsilon_1]\bigr]=[h,-\frac{1}{2} \varepsilon_1+a e]=-\frac{1}{2} \varepsilon_2
\end{equation}
Further, from the normalization condition \eqref{normcond} and formula \eqref{ee1prel} it follows that
\begin{equation}
\label{ee1e2}
\bigl[e,[\varepsilon_1,\varepsilon_2]\bigr]\in {\rm span}\{e(\lambda),h(\lambda),\varepsilon_1(\lambda)\}
\end{equation}

On the other hand, using the Jacobi identity and formulas \eqref{ee1prel},\eqref{eh},\eqref{ee2const}, we get that
\begin{equation*}
\begin{split}
~&\bigl[e,[\varepsilon_1,\varepsilon_2]\bigr]=\bigl[[e,\varepsilon_1],\varepsilon_2\bigr]+\bigl[\varepsilon_1,[e,\varepsilon_2]\bigr]=
[-\frac{1}{2} \varepsilon_1+a e, \varepsilon_2]-\\
~&\frac{1}{2} [\varepsilon_1, \varepsilon_2]\equiv -\frac{a}{2} \varepsilon_2\,\,\,{\rm mod}\,\, {\rm span}\{e(\lambda),h(\lambda),\varepsilon_1(\lambda)\},
\end{split}
\end{equation*}
which together with \eqref{ee1e2} implies that $a=0$.

\item
By analogy with the chain of the equalities \eqref{ee1e2} we can prove that
\begin{equation}
\label{eeiconst}
[e,\varepsilon_i]=-\frac{1}{2} \varepsilon_i, \quad \forall 1\leq i\leq 2m,
\end{equation}
which in turn implies by the Jacobi identity that
\begin{equation}
\label{eeiej0}
\bigl[e,[\varepsilon_i,\varepsilon_j]\bigr]=-[\varepsilon_i,\varepsilon_j], \quad \forall 1\leq i,j\leq 2m.
\end{equation}
In particular, $[e,\eta]=-\eta$.
\item
Let us show that
\begin{equation}
\label{he2mconst}
[h,\varepsilon_{2m}]=\sum_{i=1}^{m-1} (-1)^{i+1}r_i \varepsilon_{2(m-i)}
\end{equation}
From \eqref{adselfeq} and our assumptions it follows that
\begin{equation}
\label{he2mconst}
[h,\varepsilon_{2m}]=\sum_{i=1}^{m-1} (-1)^{i+1}r_i \varepsilon_{2(m-i)}+
\gamma e+\delta h
\end{equation}
for some constants $\gamma$ and $\delta$. Applying ${\rm ad}\, e$ to both sides of \eqref{he2mconst} and using the Jacobi identity and formulas \eqref{eh} and \eqref{eeiconst}, we will get that $\gamma=\delta=0$, which implies \eqref{he2mconst}.

\item
Let us prove that
\begin{equation}
\label{eiej}
[\varepsilon_i,\varepsilon_j]=d_{ij}\eta
\end{equation}
for some constants $d_{ij}$
Indeed, in general
\begin{equation}
\label{eiejprel}
[\varepsilon_i,\varepsilon_j]=b_{ij} e+c_{ij} h+d_{ij}\eta+\sum _{k=1}^{2m}a_{ij}^k\varepsilon_k
\end{equation}
where  $a_{ij}^k$, $b_{ij}$, $c_{ij}$ and $d_{ij}$ are constant by our assumptions.
Applying ${\rm ad}\, e$ to both sides of \eqref{eiejprel} and using the Jacobi identity and the formulas \eqref{eh}, \eqref{eeiconst}, and \eqref{eeiej0}, we get
\begin{equation}
\label{eiejpreljac}
-[\varepsilon_i,\varepsilon_j]= -d_{ij}\eta-\frac{1}{2}\sum _{k=1}^{2m}a_{ij}^k\varepsilon_k
\end{equation}
Comparing \eqref{eiejprel} and \eqref{eiejpreljac} we get that  $a_{ij}^k =b_{ij}=c_{ij}=0$, which implies \eqref{eiej}.

\item Moreover, by Remark \ref{symprodrem} and the definition of the vector field $\eta$ (see \eqref{ei}) the constants $d_{ij}$ from \eqref{eiej} are
either identically equal to $0$, if $i+j< 2m$ or equal to $(-1)^{i-1}$, if $i+j=2m+1$, or they are polynomial expressions (with universal constant coefficients) with respect to the constant symplectic curvatures
$r_1\ldots,r_m$, if $i+j>2m$.

\item The remaining brackets of the canonical frame are obtained iteratively from the brackets considered in the previous items.
\end{enumerate}

Therefore all nontrivial structure functions of the canonical frame are determined by the tuple $ (r_1,\ldots,r_{n-3})$, which completes the proof of uniqueness.

To prove the existence one checks by the direct computations that the models $A_{(r_1,\ldots, r_m)}$ and $D_{(r_1,\ldots, r_m)}$ have the prescribed symplectic curvatures and that all structure functions of their canonical frame are constant similarly to the proof of the existence part of Theorem 3 in \cite{doubzel2}, devoted to the computation of the canonical frame for $D_{(0,\ldots, 0)}$. $\Box$

\end{proof}

\begin{remark}
As a matter of fact it can be shown that Theorem \ref{theor1par} (with a modified set $\widetilde {\mathcal R}$), Corollary \ref{cor1}, and Theorem \ref{symmodthm90} are true if we replace the regularity condition for control systems given in  Definition \ref{regcontdef} by the following weaker one: for any point $q$ the curve of admissible velocities $\mathcal V_q$ does not belong entirely to a line through the origin. One only needs more technicalities in the description of the set  $\widetilde {\mathcal R}$ in Theorem \ref{theor1par}. $\Box$.
\end{remark}

\end{document}